\documentclass[12pt]{article}

\usepackage[margin=1in]{geometry}

\usepackage{amsmath}
\usepackage{amsfonts}
\usepackage{amssymb}
\usepackage{amsthm}
\usepackage{graphicx}
\usepackage{subfig}
\usepackage{tikz}

\newtheorem{thm}{Theorem}[section]
\newtheorem{prop}[thm]{Proposition}

\newtheorem{defn}[thm]{Definition}
\newtheorem{claim}{Claim}
\newtheorem*{provethm}{Theorem \ref{noTU}}

\theoremstyle{definition}

\theoremstyle{remark}
\newtheorem*{rem}{Remark}
\newtheorem{step}{Step}
\newtheorem{case}{Case}

\newcommand{\bb}[1]{\mathbb{#1}}

\newcommand{\ra}{\rightarrow}
\newcommand{\Ra}{\Rightarrow}
\newcommand{\minus}{\,\backslash\,}
\renewcommand{\epsilon}{\varepsilon}

\tikzstyle{vertex} = [fill=black, circle, inner sep = 0pt, minimum size = 2mm]
\tikzstyle{edge} = [->, thick, >=latex]

\begin{document}

\title{Structure theorem for $U_5$-free tournaments}
\author{Gaku Liu \\
\small Department of Mathematics \\[-0.8ex]
\small Massachusetts Institute of Technology \\[-0.8ex]
\small Cambridge, MA 02139-4307 \\
\small\tt{gakuliu@mit.edu}
}

\date{August 9, 2012}

\maketitle

\begin{abstract}
Let $U_5$ be the tournament with vertices $v_1$, \dots, $v_5$ such that $v_2 \ra v_1$, and $v_i \ra v_j$ if $j-i \equiv 1$, $2 \pmod{5}$ and $\{i,j\} \neq \{1,2\}$. In this paper we describe the tournaments which do not have $U_5$ as a subtournament. Specifically, we show that if a tournament $G$ is ``prime''---that is, if there is no subset $X \subseteq V(G)$, $1 < |X| < |V(G)|$, such that for all $v \in V(G) \minus X$, either $v \ra x$ for all $x \in X$ or $x \ra v$ for all $x \in X$---then $G$ is $U_5$-free if and only if either $G$ is a specific tournament $T_n$ or $V(G)$ can be partitioned into sets $X$, $Y$, $Z$ such that $X \cup Y$, $Y \cup Z$, and $Z \cup X$ are transitive. From the prime $U_5$-free tournaments we can construct all the $U_5$-free tournaments. We use the theorem to show that every $U_5$-free tournament with $n$ vertices has a transitive subtournament with at least $n^{\log_3 2}$ vertices, and that this bound is tight.
\end{abstract}

\section{Introduction}

A \emph{tournament} $G$ is a loopless directed graph such that for any two distinct vertices $u$, $v \in V(G)$, there is exactly one edge with both ends in $\{u,v\}$. In this paper, all tournaments are finite.
A \emph{subtournament} of a tournament $G$ is a tournament induced on a subset of $V(G)$.
For $X \subseteq V(G)$, let $G[X]$ denote the subtournament of $G$ induced on $X$.
Given two tournaments $G$ and $H$, we say that $G$ is \emph{$H$-free} if $G$ has no subtournament isomorphic to $H$; otherwise, $G$ \emph{contains} $H$.
Given a tournament $G$ and a vertex $v \in V(G)$, let $B_G(v) = \{u \in V(G) : u \ra v\}$ be the set of \emph{predecessors} of $v$ in $G$, and let $A_G(v) = \{u \in V(G) : v \ra u\}$ be the set of \emph{successors} of $v$ in $G$. 
For two disjoint sets $X$, $Y \subseteq V(G)$, we write $X \Ra Y$ if $x \ra y$ for all $x \in X$, $y \in Y$. We use $v \Ra X$ and $X \Ra v$ to mean $\{v\} \Ra X$ and $X \Ra \{v\}$, respectively.
The \emph{dual} of a tournament $G$ is the tournament obtained by reversing all edges of $G$.
A \emph{cyclic triangle} in a tournament $G$ is a set $\{v_1,v_2,v_3\} \subseteq V(G)$ of three distinct vertices such that $v_1 \ra v_2 \ra v_3 \ra v_1$. A \emph{transitive} tournament is a tournament with no cyclic triangle; a tournament $G$ is transitive if and only if its vertices can be ordered $v_1$, \dots, $v_{|V(G)|}$ such that $v_i \ra v_j$ if $i < j$. Let $I_n$ denote the transitive tournament with $n$ vertices. If the subtournament induced on a subset $X \subseteq V(G)$ is transitive, we say that $X$ is transitive.

Given a tournament $G$, a \emph{homogeneous set} of $G$ is a subset $X \subseteq V(G)$ such that for all vertices $v \in V(G) \minus X$, either $v \Ra X$ or $X \Ra v$. A homogeneous set $X \subseteq V(G)$ is \emph{nontrivial} if $1 < |X| < |V(G)|$; otherwise it is \emph{trivial}. 
A tournament is \emph{prime} if all of its homogeneous sets are trivial.
Given a tournament $G$ and a nonempty homogeneous set $X$ of $G$, let $G / X$ denote the tournament isomorphic to $G\left[ (V(G) \minus X) \cup \{v\} \right]$, where $v$ is any vertex in $X$. (Note that $G / X$ is well-defined up to isomorphism.) Thus, if $G$ has a nontrivial homogeneous set $X$, we can express it as the combination of two tournaments $G / X$ and $G[X]$ each of which has less vertices than $G$. In addition, note that if $H$ is a prime tournament, then $G$ is $H$-free if and only if $G / X$ and $G[X]$ are both $H$-free.

Define $U_5$ to be the tournament with vertices $v_1$, \dots, $v_5$ such that $v_2 \ra v_1$, and $v_i \ra v_j$ if $j-i \equiv 1$, $2 \pmod{5}$ and $\{i,j\} \neq \{1,2\}$. (Alternatively, $U_5$ is the tournament with vertices $u_1$, \dots, $u_5$ such that for any $1 \le i < j \le 5$, we have $u_i \ra u_j$ if $i$, $j$ are not both odd, and $u_j \ra u_i$ otherwise.) The tournament $U_5$ is prime. In this paper, we characterize the $U_5$-free tournaments. To do this, it suffices to characterize the prime $U_5$-free tournaments, because any tournament $G$ with a nontrivial homogeneous set $X$ is $U_5$-free if and only if the strictly smaller tournaments $G / X$ and $G[X]$ are $U_5$-free.

To state the main theorem, we define $T_n$ for odd $n \ge 1$ to be the tournament with vertices $v_1$, \dots, $v_n$ such that $v_i \ra v_j$ if $j-i \equiv 1$, 2, \dots, $(n-1)/2 \pmod{n}$. The theorem is as follows.

\begin{thm}\label{noU}
Let $G$ be a prime tournament. Then $G$ is $U_5$-free if and only if $G$ is $T_n$ for some odd $n \ge 1$ or $V(G)$ can be partitioned into sets $X$, $Y$, $Z$ such that $X \cup Y$, $Y \cup Z$, and $Z \cup X$ are transitive.
\end{thm}

The paper is organized as follows. In Section~\ref{primetournaments}, we review some results on prime tournaments and introduce the ``critical'' tournaments. In Section~\ref{examples}, we give some examples of prime $U_5$-free tournaments and verify that they satisfy Theorem~\ref{noU}. In Section~\ref{preliminaries}, we prove several preliminary facts that will be used in the proof of the main theorem. Section~\ref{mainproof} is the proof of the main theorem. In Section~\ref{transsub}, we use Theorem~\ref{noU} to show that every $U_5$-free tournament with $n$ vertices has a transitive subtournament with at least $n^{\log_3 2}$ vertices, and that this bound is tight.

\section{Prime tournaments}\label{primetournaments}

We list some properties of prime tournaments. First, note that each strong component of a tournament is a homogeneous set. Thus, if a tournament is prime, either it is strongly connected or all of its strong components have exactly one vertex. In the latter case, the tournament is transitive, and a transitive tournament is prime if and only if it has at most two vertices. Thus, every prime tournament with at least three vertices is strongly connected. 

All tournaments with at most two vertices are prime. The only prime tournament with three vertices is the tournament whose vertex set is a cyclic triangle. There are no prime tournaments with four vertices. For five vertices, there are exactly three prime tournaments $T_5$, $U_5$, and $W_5$, drawn in Figure~\ref{TUW5}. Every prime tournament with at least five vertices contains at least one of $T_5$, $U_5$, or $W_5$ (Ehrenfeucht and Rozenberg~\cite{ER}).

\begin{figure}[htbp]
\centering
\subfloat[$T_5$]{
\begin{tikzpicture}[scale=1.5]
\node[vertex] (v1) at (90:1) {};
\node[vertex] (v2) at (18:1) {};
\node[vertex] (v3) at (306:1) {};
\node[vertex] (v4) at (234:1) {};
\node[vertex] (v5) at (162:1) {};
\draw[edge] (v1) to (v2);
\draw[edge] (v1) to (v3);
\draw[edge] (v2) to (v3);
\draw[edge] (v2) to (v4);
\draw[edge] (v3) to (v4);
\draw[edge] (v3) to (v5);
\draw[edge] (v4) to (v5);
\draw[edge] (v4) to (v1);
\draw[edge] (v5) to (v1);
\draw[edge] (v5) to (v2);
\end{tikzpicture}
} \hspace{1.5cm}
\subfloat[$U_5$]{
\begin{tikzpicture}[scale=1.5]
\node[vertex] (v1) at (90:1) {};
\node[vertex] (v2) at (18:1) {};
\node[vertex] (v3) at (306:1) {};
\node[vertex] (v4) at (234:1) {};
\node[vertex] (v5) at (162:1) {};
\draw[edge] (v2) to (v1); 
\draw[edge] (v1) to (v3);
\draw[edge] (v2) to (v3);
\draw[edge] (v2) to (v4);
\draw[edge] (v3) to (v4);
\draw[edge] (v3) to (v5);
\draw[edge] (v4) to (v5);
\draw[edge] (v4) to (v1);
\draw[edge] (v5) to (v1);
\draw[edge] (v5) to (v2);
\end{tikzpicture}
} \hspace{1.5cm}
\subfloat[$W_5$]{
\begin{tikzpicture}[scale=1.5]
\node[vertex] (v1) at (90:1) {};
\node[vertex] (v2) at (18:1) {};
\node[vertex] (v3) at (306:1) {};
\node[vertex] (v4) at (234:1) {};
\node[vertex] (v5) at (162:1) {};
\draw[edge] (v2) to (v1); 
\draw[edge] (v1) to (v3);
\draw[edge] (v2) to (v3);
\draw[edge] (v2) to (v4);
\draw[edge] (v4) to (v3); 
\draw[edge] (v3) to (v5);
\draw[edge] (v4) to (v5);
\draw[edge] (v4) to (v1);
\draw[edge] (v5) to (v1);
\draw[edge] (v5) to (v2);
\end{tikzpicture}
}
\caption{The three five-vertex prime tournaments}
\label{TUW5}
\end{figure}
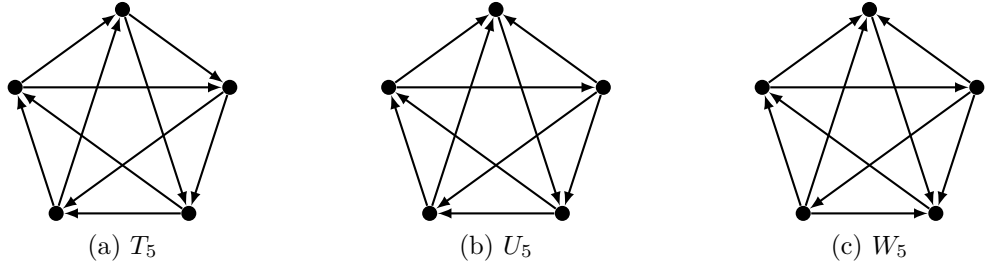

The tournaments $T_5$, $U_5$, and $W_5$ have the following generalizations to any odd number of vertices.

\begin{defn}\label{TUW}
Let $n \ge 1$ be odd. The tournaments $T_n$, $U_n$, and $W_n$ are defined as follows.
\begin{itemize}
\item $T_n$ is the tournament with vertices $v_1$, \dots, $v_n$ such that $v_i \ra v_j$ if $j-i \equiv 1$, $2$, \dots, $(n-1)/2 \pmod{n}$.
\item $U_n$ is the tournament obtained from $T_n$ by reversing all edges which have both ends in $\{v_1,\dotsc,v_{(n-1)/2}\}$.
\item $W_n$ is the tournament with vertices $v$, $w_1$, \dots, $w_{n-1}$ such that $w_i \ra w_j$ if $i < j$, and $\{w_i :  i \text{ even}\} \Ra v \Ra \{w_i : i \text{ odd}\}$.
\end{itemize}
\end{defn}

The tournaments $T_n$, $U_n$, and $W_n$ are prime for all odd $n \ge 1$, and their only prime subtournaments with at least three vertices are $T_m$, $U_m$, and $W_m$, respectively, for odd $3 \le m < n$. In addition, these tournaments are all isomorphic to their own duals. Finally, note that $T_1$, $U_1$, and $W_1$ are all the one-vertex tournament, and $T_3$, $U_3$, and $W_3$ are all the cyclic triangle tournament.

The tournaments $T_n$, $U_n$, and $W_n$ are known as the ``critical tournaments'' due to the following theorem by Schmerl and Trotter.

\begin{thm}[Schmerl and Trotter~\cite{ST}]\label{critical}
If $G$ is a prime tournament with $|V(G)| \ge 6$, and $G$ is not $T_n$, $U_n$, or $W_n$ for any odd $n$, then $G$ has a prime subtournament with $|V(G)| - 1$ vertices.
\end{thm}

In \cite{Li}, the author proved the following strengthening of Theorem~\ref{critical}.

\begin{thm}[\cite{Li}]\label{grow}
Let $G$ be a prime tournament which is not $T_n$, $U_n$, or $W_n$ for any odd $n$, and let $H$ be a prime subtournament of $G$ with $5 \le |V(H)| < |V(G)|$. Then there exists a prime subtournament of $G$ with $|V(H)|+1$ vertices that has a subtournament isomorphic to $H$.
\end{thm}

This theorem can be used to prove the following result, which appears with a different proof in Belkhechine and Boudabbous \cite{BB}.

\begin{thm}[Belkhechine and Boudabbous \cite{BB}]\label{T5}
Let $G$ be a prime tournament which contains $T_5$. Then either $G$ is $T_n$ for some odd $n$ or $G$ contains $U_5$ and $W_5$.
\end{thm}

\begin{proof}
Suppose $G$ is not $T_n$ for any $n$. In particular, since $G$ contains $T_5$, $|V(G)| \ge 6$. Since $G$ contains $T_5$, it is not $U_n$ or $W_n$ for any $n$. So by Theorem~\ref{grow}, there exists a prime subtournament $H$ of $G$, $|V(H)| = 6$, that contains $T_5$. Let $V(H) = \{u,v_1,\dotsc,v_5\}$, where $v_i \ra v_j$ if $j-i \equiv 1$ or $2 \pmod{5}$.

Let $B = B_H(u)$ and $A = A_H(u)$. By taking the dual of $G$ if necessary (we can do this because $T_5$, $U_5$, and $W_5$ are isomorphic to their own duals), we can assume $|B| \le 2$. 
If $|B| = 0$, then $u \Ra V(H) \minus \{u\}$, so $V(H) \minus \{u\}$ is a homogeneous set in $H$, contradicting the primeness of $H$. 
Suppose $|B| = 1$. Without loss of generality, assume $v_1 \in B$ and $v_2$, \dots, $v_5 \in A$. Then the tournament induced on $\{v_2,u,v_4,v_5,v_1\}$ is isomorphic to $U_5$, and the tournament induced on $\{v_5,v_1,u,v_2,v_3\}$ is isomorphic to $W_5$, as desired.

Finally, suppose $|B| = 2$. We either have $v_i$, $v_{i+1} \in B$ for some $i$ or $v_i$, $v_{i+2} \in B$ for some $i$, where the indices are taken modulo 5. Suppose we have the former case; without loss of generality, assume $i=1$. Then $\{u,v_3\}$ is a homogeneous set in $H$, which contradicts the primeness of $H$. So we must have $v_i$, $v_{i+2} \in B$ for some $i$; without loss of generality, assume $i=1$. Then the tournament induced on $\{v_2,v_1,v_3,u,v_5\}$ is isomorphic to $U_5$, and the tournament induced on $\{v_2,v_3,u,v_4,v_5\}$ is isomorphic to $W_5$.
\end{proof}

Theorem~\ref{T5} implies that to prove Theorem~\ref{noU}, it suffices to prove the following.

\begin{thm}\label{noTU}
Let $G$ be a prime tournament. Then $G$ is $T_5$-free and $U_5$-free if and only if $V(G)$ can be partitioned into sets $X$, $Y$, $Z$ such that $X \cup Y$, $Y \cup Z$, and $Z \cup X$ are transitive.
\end{thm}

To conclude this section, we state a forbidden-subtournament result different from ours. In \cite{La}, Latka characterized all the $W_5$-free tournaments. To state the theorem, we define $Q_7$ to be the Paley tournament on 7 vertices; that is, the tournament with vertices $v_1$, \dots, $v_7$ such that $v_i \ra v_j$ if $j-i$ is a quadratic residue modulo 7. Let $Q_7 - v$ be the tournament obtained from $Q_7$ by deleting a vertex.

\begin{thm}[Latka \cite{La}]\label{noW}
A prime tournament is $W_5$-free if and only if it is isomorphic to one of $I_0$, $I_2$, $Q_7 - v$, $Q_7$, $T_n$, or $U_n$ for some some odd $n \ge 1$. 
\end{thm}

\section{Examples}\label{examples}

We now give some examples of families of prime $U_5$-free tournaments. First, we have that $T_n$ is $U_5$-free for all odd $n \ge 1$; this obviously agrees with Theorem~\ref{noU}.

The tournament $W_n$ is $U_5$-free for for all odd $n \ge 1$. To see that this agrees with Theorem~\ref{noU}, let $V(W_n) = \{v,w_1,\dotsc,w_{n-1}\}$ as in Definition~\ref{TUW}, and let $X = \{v\}$, $Y = \{w_i : i \text{ odd}\}$, and $Z = \{w_i : i \text{ even}\}$. Then $X$, $Y$, and $Z$ partition $V(W_n)$, and $X \cup Y$, $Y \cup Z$, and $Z \cup X$ are transitive, as desired.

Finally, let $P_n$ be the tournament with vertices $v_1$, \dots, $v_n$ such that $v_i \ra v_j$ if $j - i \ge 2$, and $v_{i+1} \ra v_i$ for all $1 \le i < n$. Then $P_n$ is prime for all $n \neq 4$, and $P_n$ is $U_5$-free for all $n$. Let $X = \{v_i : i \equiv 0 \pmod{3}\}$, $Y = \{v_i : i \equiv 1 \pmod{3}\}$, and $Z = \{v_i : i \equiv 2 \pmod{3}\}$. Then $X$, $Y$, and $Z$ partition $V(P_n)$, and $X \cup Y$, $Y \cup Z$, and $Z \cup X$ are transitive.

\section{Preliminaries}\label{preliminaries}

Before proving the main theorem, we establish some facts that will aid us in the proof.

\begin{prop}\label{twotransitive}
Let $G$ be a tournament, and let $X$, $Y \subseteq V(G)$ be disjoint transitive sets. Let $X = \{x_1,\dotsc,x_\ell\}$ and $Y = \{y_1,\dotsc,y_m\}$, where $x_i \ra x_j$ if $i < j$ and $y_i \ra y_j$ if $i < j$. Suppose that for each $x_i \in X$, there is an integer $1 \le s_i \le m+1$ such that $\{y_j : j < s_i\} \Ra x_i \Ra \{y_j : j \ge s_i\}$. Then $X \cup Y$ is transitive if and only if $s_1 \le s_2 \le \dotsc \le s_\ell$.
\end{prop}

Proposition~\ref{twotransitive} is clear; we omit the proof.

\begin{prop}\label{sequence}
Let $G$ be a strongly connected tournament, $|V(G)|  \ge 3$, such that $V(G)$ can be partitioned into sets $X$, $Y$, $Z$ such that $X \cup Y$, $Y \cup Z$, and $Z \cup X$ are transitive. Let $X = \{x_1,\dotsc,x_\ell\}$, $Y = \{y_1,\dotsc,y_m\}$, and $Z = \{z_1,\ldots,z_n\}$, where $x_i \ra x_j$ if $i < j$, and similarly for $Y$ and $Z$. Then there exists a sequence $C_1$, $C_2$, \dots, $C_{|V(G)|-2}$ of cyclic triangles such that
$C_1 = \{x_1,y_1,z_1\}$
and if $C_r = \{x_i,y_j,z_k\}$, $1 \le r < |V(G)| - 2$, then $C_{r+1} = \{x_{i+1},y_j,z_k\}$, $\{x_i,y_{j+1},z_k\}$, or $\{x_i,y_j,z_{k+1}\}$.
\end{prop}

Note that if $C_1$, \dots, $C_{|V(G)|-2}$ is such a sequence, then $V(G) = C_1 \cup \dotsb \cup C_{|V(G)|-2}$.

\begin{proof}
First, note that if any of $X$, $Y$, or $Z$ are empty, then $G$ is transitive, which contradicts the fact that $G$ is strongly connected since $|V(G)| \ge 3$. So $X$, $Y$, and $Z$ are all nonempty.

It suffices to show that $\{x_1,y_1,z_1\}$ is a cyclic triangle, and if $\{x_i,y_j,z_k\}$ is a cyclic triangle and $i+j+k < |V(G)|$, then one of $\{x_{i+1},y_j,z_k\}$, $\{x_i,y_{j+1},z_k\}$, or $\{x_i,y_j,z_{k+1}\}$ exists (e.g., $\{x_{i+1},y_j,z_k\}$ exists if and only if $i < \ell$) and is a cyclic triangle.
First, suppose that $\{x_1,y_1,z_1\}$ is not a cyclic triangle. Then one of the vertices in $\{x_1,y_1,z_1\}$ is a predecessor of the other two; without loss of generality, assume $x_1 \Ra \{y_1,z_1\}$. Then since $X \cup Y$ and $X \cup Z$ are transitive, we have $x_1 \Ra V(G) \minus \{x\}$. This contradicts the strong connectivity of $G$. So $\{x_1,y_1,z_1\}$ is a cyclic triangle.

Finally, suppose that $\{x_i,y_j,z_k\}$ is a cyclic triangle and $i+j+k < |V(G)|$. Without loss of generality, assume $x_i \ra y_j \ra z_k \ra x_i$. Suppose that none of $\{x_{i+1},y_j,z_k\}$, $\{x_i,y_{j+1},z_k\}$, and $\{x_i,y_j,z_{k+1}\}$ are cyclic triangles (if they exist). If $i < \ell$, then since $z_k \ra x_i$ and $Z \cup X$ is transitive, we have $z_k \ra x_{i+1}$. Then since $\{x_{i+1},y_j,z_k\}$ is not a cyclic triangle and $y_j \ra z_k$, we have $y_j \ra x_{i+1}$ as well. So $\{x_i,y_j,z_k\} \Ra \{x_{i+1}\}$. Since $X \cup Y$ and $X \cup Z$ are transitive, we thus have
\[
\{x_{i^\prime} : i^\prime \le i\} \cup \{y_{j^\prime} : j^\prime \le j\} \cup \{z_{k^\prime} : k^\prime \le k\} \Ra \{x_{i^\prime} : i^\prime > i\}.
\]
The above statement also holds if $i=\ell$, since in that case the right side is empty.

Now, we similarly have $\{x_{i^\prime} : i^\prime \le i\} \cup \{y_{j^\prime} : j^\prime \le j\} \cup \{z_{k^\prime} : k^\prime \le k\} \Ra \{y_{j^\prime} : j^\prime > j\}$ and $\{x_{i^\prime} : i^\prime \le i\} \cup \{y_{j^\prime} : j^\prime \le j\} \cup \{z_{k^\prime} : k^\prime \le k\} \Ra \{z_{k^\prime} : k^\prime > k\}$. Thus,
\[
\{x_{i^\prime} : i^\prime \le i\} \cup \{y_{j^\prime} : j^\prime \le j\} \cup \{z_{k^\prime} : k^\prime \le k\} \Ra \{x_{i^\prime} : i^\prime > i\} \cup \{y_{j^\prime} : j^\prime > j\} \cup \{z_{k^\prime} : k^\prime > k\}.
\]
Since $i + j + k < |V(G)|$, the right side of the above statement is nonempty. This contradicts the strong connectivity of $G$. Thus, one of $\{x_{i+1},y_j,z_k\}$, $\{x_i,y_{j+1},z_k\}$, or $\{x_i,y_j,z_{k+1}\}$ is a cyclic triangle, completing the proof.
\end{proof}

\begin{prop}\label{lemma}
Let $G$ be a tournament which is $T_5$-free and $U_5$-free, and let $v \in V(G)$. Let $B = B_G(v)$ and $A = A_G(v)$. Suppose that $C = \{x,y,z\} \subseteq V(G) \minus \{v\}$ is a cyclic triangle. Then the following hold.
\renewcommand{\theenumi}{(\alph{enumi})}
\renewcommand{\labelenumi}{\theenumi}
\begin{enumerate}
\item\label{BBA} If $|C \cap A| = 1$ and there is $u \in V(G) \minus \{v\}$ such that $u \Ra C$, then $u \in B$.
\item\label{BAA} If $|C \cap A| = 2$ and there is $u \in V(G) \minus \{v\}$ such that $C \Ra u$, then $u \in A$.
\newcounter{enumisave}
\setcounter{enumisave}{\value{enumi}}
\end{enumerate}
If in addition we have $x^\prime \in V(G) \minus \{v,x,y,z\}$ such that $C^\prime = \{x^\prime,y,z\}$ is a cyclic triangle and $x \ra x^\prime$, then the following hold.
\begin{enumerate}
\setcounter{enumi}{\value{enumisave}}
\item\label{BBB} If $|C \cap A| > 0$, then $|C^\prime \cap A| > 0$.
\item\label{AAA} If $|C \cap A| = 3$, then $|C^\prime \cap A| = 3$.
\item\label{BA} If $y \ra z$ and $y \in B$, $z \in A$, then $x$ and $x^\prime$ are either both in $B$ or both in $A$.
\end{enumerate}
\end{prop}

\begin{proof}
Assume without loss of generality that $x \ra y \ra z \ra x$. For part \ref{BBA}, assume without loss of generality that $x \in A$ and $y$, $z \in B$. Suppose $u \in V(G) \minus \{v\}$ such that $u \Ra C$. If $u \in A$, then the tournament induced on $\{x,u,y,z,v\}$ is $U_5$, a contradiction. So $u \in B$. Part \ref{BAA} follows from \ref{BBA} by taking the dual.

For parts \ref{BBB} through \ref{BA}, we have $x^\prime \ra y \ra z \ra x^\prime$. For part (c), suppose that $|C \cap A| > 0$ and $|C^\prime \cap A| = 0$. Then we must have $x \in A$ and $x^\prime$, $y$, $z \in B$. But then the tournament induced on $\{v,z,x,x^\prime,y\}$ is $U_5$, a contradiction. This proves \ref{BBB}. Part \ref{AAA} follows from the contrapositive of \ref{BBB} by taking the dual.

Finally, for part \ref{BA}, suppose $y \in B$ and $z \in A$. First, suppose that $x \in B$ and $x^\prime \in A$. Then the tournament induced on $\{x^\prime, x, y, v, z\}$ is $U_5$, a contradiction. Now suppose that $x \in A$ and $x^\prime \in B$. Then the tournament induced on $\{x, x^\prime, y, v, z\}$ is $T_5$, a contradiction. So $x$ and $x^\prime$ are either both in $B$ or both in $A$, as desired.
\end{proof}

In particular, given a sequence $C_1$, \dots, $C_{|V(G)|-2}$ of cyclic triangles as in Proposition~\ref{sequence}, we can apply Proposition~\ref{lemma}\ref{BBB}-\ref{BA} to consecutive cyclic triangles $C_r$, $C_{r+1}$. 

\section{Proof of theorem}\label{mainproof}

We now prove Theorem~\ref{noU}. Recall that by Theorem~\ref{T5}, it suffices to prove the following.

\begin{provethm}
Let $G$ be a prime tournament. Then $G$ is $T_5$-free and $U_5$-free if and only if $V(G)$ can be partitioned into sets $X$, $Y$, $Z$ such that $X \cup Y$, $Y \cup Z$, and $Z \cup X$ are transitive.
\end{provethm}

\begin{proof}
We organize the proof into three steps.

\begin{step}
``If'' direction
\end{step}

Suppose $V(G)$ can be partitioned into sets $X$, $Y$, $Z$ such that $X \cup Y$, $Y \cup Z$, and $Z \cup X$ are transitive. Then for any five-vertex subtournament $H$ of $G$, at least one of $|V(H) \cap X|$, $|V(H) \cap Y|$, $|V(H) \cap Z|$ is at most 1, so at least one of $|V(H) \cap (X \cup Y)|$, $|V(H) \cap (Y \cup Z)|$, $|V(H) \cap (Z \cup X)|$ is at least 4. Hence, $H$ has a four-vertex transitive subtournament, and thus cannot be $T_5$ or $U_5$. So $G$ is $T_5$-free and $U_5$-free, proving one direction of the theorem.

\begin{step}
Setting up the induction
\end{step}

We now prove the other direction. Suppose $G$ is $T_5$-free and $U_5$-free.
We wish to prove that $V(G)$ can be partitioned into sets $X$, $Y$, $Z$ such that $X \cup Y$, $Y \cup Z$, and $Z \cup X$ are transitive.
We proceed by induction on $|V(G)|$.
If $|V(G)| \le 5$, then since $G$ is prime and is $T_5$-free and $U_5$-free, we have either $|V(G)| \le 2$ or $G$ is $W_n$ for $n = 3$ or 5. In the former case, let $X = V(G)$ and $Y$, $Z = \varnothing$; in the latter case, let $X = \{v\}$, $Y = \{w_i : i \text{ odd}\}$, and $Z = \{w_i : i \text{ even}\}$, where $v$, $w_1$, \dots, $w_n$ are as in Definition~\ref{TUW}. In either case, $X \cup Y$, $Y \cup Z$, and $Z \cup X$ are transitive, as desired.

Now, assume $|V(G)| \ge 6$, and that the theorem holds for all prime tournaments $G^\prime$ with $|V(G^\prime)| < |V(G)|$. If $G$ is $W_n$ for some $n$, then we are done by setting $X = \{v\}$, $Y = \{w_i : i \text{ odd}\}$, and $Z = \{w_i : i \text{ even}\}$ as before. Assume $G$ is not $W_n$ for any $n$. Since $G$ is $T_5$-free and $U_5$-free, it is not $T_n$ or $U_n$ for any $n$.
Thus, by Theorem~\ref{critical}, there is a prime subtournament $G^\prime$ of $G$ with $|V(G^\prime)| = |V(G)| - 1$.
Now, $G^\prime$ is $T_5$-free and $U_5$-free, so by the inductive hypothesis we can partition $V(G^\prime)$ into sets $X$, $Y$, and $Z$ such that $X \cup Y$, $Y \cup Z$, and $Z \cup X$ are transitive. Let $X = \{x_1,\dotsc,x_\ell\}$, $Y = \{y_1,\dotsc,y_m\}$, and $Z = \{z_1,\dotsc,z_n\}$ such that $x_i \ra x_j$ if $i < j$, and similarly for $Y$ and $Z$.
Since $G^\prime$ is strongly connected (because $G^\prime$ is prime and $|V(G^\prime)| \ge 5$), there exists a sequence $C_1$, $C_2$, \dots, $C_{|V(G^\prime)|-2}$ of cyclic triangles in $G^\prime$ as in Theorem~\ref{sequence}.

Let $v$ be the vertex such that $\{v\} = V(G) \minus V(G^\prime)$. For convenience, let $B = B_G(v)$ and $A = A_G(v)$; thus, $V(G^\prime) = B \cup A$.
Now, $V(G^\prime) = C_1 \cup \dotsb \cup C_{|V(G^\prime)|-2}$. Suppose that for all $1 \le r \le |V(G)| - 2$, we have either $C_r \subseteq B$ or $C_r \subseteq A$. Then since $C_r \cap C_{r+1} \neq \varnothing$ for all $1 \le r < |V(G^\prime)|-2$, we must have either $V(G^\prime) \subseteq B$ or $V(G^\prime) \subseteq A$. But then $V(G^\prime)$ is a nontrivial homogeneous set of $G$, contradicting the primeness of $G$. Thus, there is some $1 \le r \le |V(G^\prime)|-2$ such that neither $C_r \subseteq B$ nor $C_r \subseteq A$; i.e., there is some $r$ such that  $|C_r \cap A| = 1$ or 2. By taking the dual of $G$ if necessary, we may assume there is some $r$ such that $|C_r \cap A| = 2$. Choose $r_0$ to be the minimum $r$ such that $|C_r \cap A| = 2$. Let $C_{r_0} = \{x_{i_0}, y_{j_0}, z_{k_0}\}$, and without loss of generality, assume 
\begin{itemize}
\item $x_{i_0} \ra y_{j_0} \ra z_{k_0} \ra x_{i_0}$, and 
\item $x_{i_0} \in A$, $y_{j_0} \in A$, and $z_{k_0} \in B$.
\end{itemize}

We set some final notation before proceeding. 
\begin{itemize}
\item Let $s$ be the smallest integer such that $y_s \in A$ (thus, $s \le j_0$), and 
\item let $t$ be the largest integer such that $z_{t-1} \in B$ (thus, $t \ge k_0+1$). 
\end{itemize}
In addition, since $X \cup Y$ and $X \cup Z$ are transitive, for each $x_i \in X$ 
\begin{itemize}
\item let $1 \le s_i \le m+1$ be the integer such that $\{y_j : j < s_i\} \Ra x_i \Ra \{y_j : j \ge s_i\}$, and 
\item let $1 \le t_i \le n+1$ be the integer such that $\{z_k : k < t_i\} \Ra x_i \Ra \{z_k : k \ge t_i\}$. 
\end{itemize}
Note that $s_1 \le \dotsb \le s_\ell$ and $t_1 \le \dotsb \le t_\ell$.

\begin{step}
Completing the induction
\end{step}

We claim that for the partition $\left\{ X \cup \{v\}, Y, Z \right\}$ of $V(G)$, the sets $X \cup \{v\} \cup Y$, $X \cup \{v\} \cup Z$, and $Y \cup Z$ are transitive, which will prove the theorem. To prove this claim, it suffices (by Proposition~\ref{twotransitive}) to show the following.
\renewcommand{\theenumi}{(\arabic{enumi})}
\renewcommand{\labelenumi}{\theenumi}
\begin{enumerate}
\item\label{Xtrans} $\{x_i : i < i_0\} \Ra v \Ra \{x_i : i \ge i_0\}$.
\item\label{Ytrans} $\{y_j : j < s\} \Ra v \Ra \{y_j : j \ge s\}$.
\item\label{Ztrans} $\{z_k : k < t\} \Ra v \Ra \{z_k : k \ge t\}$.
\item\label{Yfit} $s_{i_0-1} \le s \le s_{i_0}$ (if $i_0 = 1$, only the right-hand inequality needs to be proved).
\item\label{Zfit} $t_{i_0-1} \le t \le t_{i_0}$ (if $i_0 = 1$, only the right-hand inequality needs to be proved).
\end{enumerate}
We prove these statements through a series of claims.

\begin{claim}\label{direction}
For all $1 \le r \le |V(G^\prime)|-2$, if $C_r = \{x_i,y_j,z_k\}$, then $x_i \ra y_j \ra z_k \ra x_i$.
\end{claim}

This follows from the assumption $x_{i_0} \ra y_{j_0} \ra z_{k_0} \ra x_{i_0}$ and the properties of the sequence $C_1$, \dots, $C_{|V(G^\prime)|-2}$.

\begin{claim}\label{before}
Suppose $C_r = \{x_i,y_j,z_k\}$ for some $r$ and $x_i \in B$, $y_j \in A$, and $z_k \in B$. Then $s_i \le s$ and $t_i \le t$.
\end{claim}

Suppose $s_i > s$. Then by the definition of $s_i$, we have $y_s \ra x_i$. Since $x_i \ra y_j$ and $X \cup Y$ is transitive, we also have $y_s \ra y_j$. Then since $y_j \ra z_k$ and $Y \cup Z$ is transitive, we have $y_s \ra z_k$. So $y_s \Ra C_r$. But $|C_r \cap A| = 1$, and by the definition of $s$, we have $y_s \in A$. This contradicts Proposition~\ref{lemma}\ref{BBA}. So we must have $s_i \le s$.

Now suppose that $t_i > t$. In particular, this means $t \le t_i-1 \le n$, so the vertex $z_t$ exists. By the definition of $t_i$, we have $z_t \ra x_i$. Also, since $z_k \in B$, by the definition of $t$ we have $k < t$, so $z_k \ra z_t$. Finally, since $y_j \ra z_k$ and $Y \cup Z$ is transitive, we have $y_j \ra z_t$. Thus, $\{z_k,x_i,y_j\}$ and $\{z_t,x_i,y_j\}$ are cyclic triangles with $z_k \ra z_t$. However, $x_i \in B$, $y_j \in A$, $z_k \in B$, and by the definition of $t$, $z_t \in A$. This contradicts Proposition~\ref{lemma}\ref{BA}. So $t_i \le t$, as desired.

\begin{claim}\label{after}
Suppose $C_r = \{x_i,y_j,z_k\}$ for some $r$ and $x_i \in A$, $y_j \in A$, and $z_k \in B$. Then $s_i \ge s$ and $t_i \ge t$.
\end{claim}

Suppose $s_i < s$. In particular, we have $s-1 \ge s_i \ge 1$, so the vertex $y_{s-1}$ exists. By the definition of $s_i$, $x_i \ra y_{s-1}$. Also, since $y_j \in A$, by the definition of $s$ we have $s-1 < j$, so $y_{s-1} \ra y_j$. Finally, since $y_j \ra z_k$ and $Y \cup Z$ is transitive, we have $y_{s-1} \ra z_k$. Thus, $\{y_{s-1},z_k,x_i\}$ and $\{y_j,z_k,x_i\}$ are cyclic triangles with $y_{s-1} \ra y_j$. However, $z_k \in B$, $x_i \in A$, $y_j \in A$, and by the definition of $s$, $y_{s-1} \in B$. This contradicts Proposition~\ref{lemma}\ref{BA}. So $s_i \ge s$.

Now suppose $t_i <  t$. By the definition of $t_i$, we have $x_i \ra z_{t-1}$. Since $z_k \ra x_i$ and $Z \cup X$ is transitive, we also have $z_k \ra z_{t-1}$. Since $y_j \ra z_k$ and $Y \cup Z$ is transitive, we then have $y_j \ra z_{t-1}$. So $C_r \Ra z_{t-1}$. But $|C_r \cap A| = 2$, and by the definition of $t$, we have $z_{t-1} \in B$. This contradicts Proposition~\ref{lemma}\ref{BAA}. So $t_i \ge t$, as desired.

\begin{claim}\label{r0minus1}
If $r_0 > 1$, then $C_{r_0-1} = \{x_{i_0-1}, y_{j_0}, z_{k_0}\}$, and $x_{i_0-1} \in B$, $y_{j_0} \in A$, and $z_{k_0} \in B$.
\end{claim}

Assume $r_0 > 1$. Recall that by assumption, $x_{i_0} \in A$, $y_{j_0} \in A$, and $z_{k_0} \in B$. By the defintion of $r_0$, we must have $|C_{r_0-1}| \neq 2$. So one of the following holds.
\begin{itemize}
\item $C_{r_0-1} = \{x_{i_0-1}, y_{j_0}, z_{k_0}\}$ and $x_{i_0} \in B$.
\item $C_{r_0-1} = \{x_{i_0}, y_{j_0-1}, z_{k_0}\}$ and $y_{j_0-1} \in B$.
\item $C_{r_0-1} = \{x_{i_0}, y_{j_0}, z_{k_0-1}\}$ and $z_{j_0-1} \in A$.
\end{itemize}
If we have the first case then we are done. The second case contradicts Proposition~\ref{lemma}\ref{BA}. The third case contradicts Proposition~\ref{lemma}\ref{AAA}. This proves the claim.

\begin{claim}\label{down}
For all $1 \le r < r_0$, if $C_r = \{x_i,y_j,z_k\}$, then $x_i \in B$ and $z_k \in B$. In addition, $y_j \in A$ if $j \ge s$ and $y_j \in B$ otherwise.
\end{claim}

If $r_0 = 1$, there is nothing to prove. So assume $r_0 > 1$. We prove the claim by downward induction on $r$. If $r = r_0 - 1$, then the conclusion follows by Claim~\ref{r0minus1} and the fact that $j_0 \ge s$. Now suppose the conclusion holds for some $2 \le r < r_0$. Let $C_r = \{x_i,y_j,z_k\}$. First suppose that $j < s$. Then $|C_r \cap A| = 0$ by the inductive hypothesis. Thus, by Proposition~\ref{lemma}\ref{BBB} we have $|C_{r-1} \cap A| = 0$ as well, as desired.

Now suppose that $j \ge s$. By the inductive hypothesis, $x_i \in B$, $y_j \in A$, and $z_k \in B$. We have three cases.
\begin{itemize}
\item $C_{r-1} = \{x_{i-1}, y_j, z_k\}$.
\item $C_{r-1} = \{x_i, y_{j-1}, z_k\}$.
\item $C_{r-1} = \{x_i, y_j, z_{k-1}\}$.
\end{itemize}
For the first case, since $|C_{r-1} \cap A| \neq 2$ by the definition of $r_0$, we must have $x_{i-1} \in B$, as desired. Similarly, in the third case we must have $z_{k-1} \in B$, as desired. 

Finally, suppose we have the second case. We wish to prove that $y_{j-1} \in B$ if $j=s$, and $y_{j-1} \in A$ if $j > s$. If $j = s$, then by the definition of $s$ we must have $y_{j-1} \in B$, as desired. 
Now suppose $j > s$. Suppose $y_{j-1} \in B$. Then $|C_{r-1} \cap A| = 0$, so by Proposition~\ref{lemma}\ref{BBB} and downward induction, we have $|C_{r^\prime} \cap A| = 0$ for all $r^\prime \le r-1$. However, then $y_{j^\prime} \in B$ for all $j^\prime \le j-1$, and in particular $y_s \in B$ since $j > s$. This contradicts the fact that $y_s \in A$ by the definition of $s$. Thus, $y_{j-1} \in A$, as desired.

\begin{claim}\label{up}
For all $r_0 \le r \le |V(G^\prime)|-2$, if $C_r = \{x_i,y_j,z_k\}$, then $x_i \in A$ and $y_j \in A$. In addition, $z_k \in B$ if $k < t$ and $z_k \in A$ otherwise.
\end{claim}

We induct upwards on $r$. If $r = r_0$, the conclusion holds by assumption and the fact that $k_0+1 \le t$. Now suppose the conclusion holds for some $r_0 \le r \le |V(G^\prime)|-3$. Let $C_r = \{x_i,y_j,z_k\}$. First suppose $k \ge t$. Then $|C_r \cap A| = 3$ by the inductive hypothesis, so by Proposition~\ref{lemma}\ref{AAA}, we have $|C_{r+1} \cap A| = 3$, as desired.

Now suppose that $k < t$. By the inductive hypothesis, $x_i \in A$, $y_j \in A$, and $z_k \in B$. We have three cases.
\begin{itemize}
\item $C_{r+1} = \{x_{i+1}, y_j, z_k\}$.
\item $C_{r+1} = \{x_i, y_{j+1}, z_k\}$.
\item $C_{r+1} = \{x_i, y_j, z_{k+1}\}$.
\end{itemize}
For the second case, by Proposition~\ref{lemma}\ref{BA} we must have $y_{j+1} \in A$, as desired. Now suppose we have the third case. We wish to prove that $z_{k+1} \in A$ if $k = t-1$, and $z_{k+1} \in B$ if $k < t-1$.
If $k = t-1$, then by the definition of $t$ we have $z_{k+1} \in A$, as desired. Now suppose $k < t-1$, and suppose that $z_{k+1} \in A$. Then $|C_{r+1} \cap A| = 3$, so by Proposition~\ref{lemma}\ref{AAA} and upward induction, we have $|C_{r^\prime} \cap A| = 3$ for all $r^\prime \ge r+1$. Then $z_{k^\prime} \in A$ for all $k^\prime \ge k+1$, so in particular $z_{t-1} \in A$ since $k < t-1$. This contradicts the fact that $z_{t-1} \in B$ by the definition of $t$. Thus, $z_{k+1} \in B$, as desired.

Finally, suppose we have the first case. We wish to prove $x_{i+1} \in A$. Suppose that $x_{i+1} \in B$. Applying Claim~\ref{before} to $C_{r+1}$, we have $s_{i+1} \le s$ and $t_{i+1} \le t$. However, applying Claim~\ref{after} to $C_r$, we have $s_i \ge s$ and $t_i \ge t$. Since $s_i \le s_{i+1}$ and $t_i \le t_{i+1}$ by Proposition~\ref{twotransitive}, we must therefore have $s_i = s_{i+1} = s$ and $t_i = t_{i+1} = t$. It follows that
\begin{align*}
\{x_{i^\prime} : i^\prime < i\} &\Ra \{x_i, x_{i+1}\} \Ra \{x_{i^\prime} : i^\prime > i+1\}, \\
\{y_{j^\prime} : j^\prime < s\} &\Ra \{x_i, x_{i+1}\} \Ra \{y_{j^\prime} : j^\prime \ge s\}, \\
\{z_{k^\prime} : k^\prime < t\} &\Ra \{x_i, x_{i+1}\} \Ra \{z_{k^\prime} : k^\prime \ge t\}
\end{align*}
and hence $\{x_i,x_{i+1}\}$ is a homogeneous set in $G^\prime$. This contradicts the fact that $G^\prime$ is prime. So we must have $x_{i+1} \in A$, as desired.

\begin{claim}
(1) through (5) hold.
\end{claim}

\ref{Xtrans}, \ref{Ytrans}, and \ref{Ztrans} follow by Claims~\ref{down} and \ref{up}. For $i_0 > 1$, the lower bounds of \ref{Yfit} and \ref{Zfit} follow by Claims~\ref{r0minus1} and \ref{before}. The upper bounds of \ref{Yfit} and \ref{Zfit} follow by Claim \ref{after}. This completes the proof of the theorem.
\end{proof}

\section{Large transitive subtournaments}\label{transsub}

The tournament version of the Erd\H{o}s-Hajnal conjecture states that for every tournament $H$, there exists a constant $\epsilon > 0$ such that every $H$-free tournament with $n$ vertices has a transitive subtournament with at least $n^\epsilon$ vertices. (This conjecture was stated in \cite{APS} and proven there to be equivalent to the original undirected graph version of the conjecture.) Berger, Choromanski, and Chudnovsky proved in \cite{BCC} that every tournament with at most five vertices satisfies the Erd\H{o}s-Hajnal conjecture. Here, we give a precise result for $U_5$.

\begin{thm}\label{log32}
If $G$ is a $U_5$-free tournament, then $G$ has a transitive subtournament with at least $|V(G)|^{\log_3 2}$ vertices. In addition, there are infinitely many $U_5$-free tournaments $G$ which have no transitive subtournament with more than $|V(G)|^{\log_3 2}$ vertices. 
\end{thm}

\begin{rem}
For a tournament $H$, let $\xi(H)$ denote the supremum of all $\epsilon$ for which there exists $c >0$ such that every $H$-free tournament $G$ has a transitive subtournament with at least $c |V(G)|^\epsilon$ vertices. Then Theorem~\ref{log32} implies that $\xi(U_5) = \log_3 2$. In \cite{BCCF} and \cite{CCS}, all the tournaments $H$ with $\xi(H) = 1$ were characterized, and it was shown that there are no tournaments $H$ with $5/6 < \xi(H) < 1$.
\end{rem}

Before proving Theorem~\ref{log32}, we introduce the following terminology. 

\begin{defn}\label{substitute}
Let $G^\prime$ be a tournament, and fix an ordering $(v_1,\dotsc,v_m)$ of its vertices. Then for non-null tournaments $H_1$, \dots, $H_m$, let $G^\prime(H_1,\dotsc,H_m)$ denote the tournament with vertex set $V_1 \cup \dotsb \cup V_m$ such that
\begin{itemize}
\item $V_1$, \dots, $V_m$ are pairwise disjoint,
\item for all $i$, the subtournament of $G^\prime(H_1,\dotsc,H_m)$ induced on $V_i$ is isomorphic to $H_i$, and
\item $V_i \Ra V_j$ in $G^\prime(H_1,\dotsc,H_m)$ if $v_i \ra v_j$ in $G^\prime$.
\end{itemize}
\end{defn}

Every tournament $G$ with $|V(G)| > 1$ can be written as $G^\prime(H_1,\dotsc,H_m)$ for some prime tournament $G^\prime$ with $|V(G^\prime)| > 1$. Moreover, if $H$ is a prime tournament, then $G^\prime(H_1,\dotsc,H_m)$ is $H$-free if and only if $G^\prime$, $H_1$, \dots, $H_m$ are all $H$-free.

Our proof of Theorem~\ref{log32} will rely on the classical \emph{Karamata's inequality} for concave functions, which we state below for convenience.

\begin{thm}[Karamata]\label{karamata}
Let $f$ be a real-valued, concave function defined on an interval $I \subseteq \bb{R}$. Suppose $x_1$, \dots, $x_n$, $y_1$, \dots, $y_n \in I$ such that
\begin{itemize}
\item $x_1 \ge \dotsb \ge x_n$ and $y_1 \ge \dotsb \ge y_n$, 
\item $x_1 + \dotsb + x_i \le y_1 + \dotsb + y_i$ for all $1 \le i < n$, 
\item $x_1 + \dotsb + x_n = y_1 + \dotsb + y_n$.
\end{itemize}
Then
\[
f(x_1) + \dotsb + f(x_n) \ge f(y_1) + \dotsb + f(y_n).
\]
\end{thm}

We use this inequality to prove the following.

\begin{prop}\label{2outof3}
Let $\gamma = \log_3 2$. If $a$, $b$, $c$ are nonnegative real numbers with $c = \min(a,b,c)$ and $a+b+c=n$, then
\[
a^\gamma + b^\gamma \ge n^\gamma.
\]
\end{prop}

\begin{proof}
Note that the function $f(x) = x^\gamma$ is concave on the interval $[0,\infty)$. Now, we have $a$,~$b \ge c$ and $a+b+c = n$, and hence $a$, $b \le n-2c$. So applying Karamata's inequality to the function $f$, we have
\[
a^\gamma + b^\gamma \ge (n-2c)^\gamma + c^\gamma.
\]
Let $g(x) = (n-2x)^\gamma + x^\gamma$. Since the functions $(n-2x)^\gamma$ and $x^\gamma$ are both concave on $[0,n/3]$, we have that $g$ is concave on this interval. Hence, the minimum value of $g$ on $[0,n/3]$ occurs at $x = 0$ or $x = n/3$. We have $g(0) = g(n/3) = n^\gamma$, so the minimum of $g$ on $[0,n/3]$ is $n^\gamma$. We have $c \in [0,n/3]$, so
\[
(n-2c)^\gamma + c^\gamma \ge n^\gamma
\]
 which gives the desired inequality.
\end{proof}

We are now ready to prove Theorem~\ref{log32}.

\begin{proof}[Proof of Theorem~\ref{log32}] 
For convenience, let $\gamma = \log_3 2$.

We first show that the bound $n^\gamma$ is tight. For $n \ge 1$, define the tournaments $G_n$ inductively by $G_1 = T_3$, and $G_{n+1} = T_3(G_n,G_n,G_n)$. Then $G_n$ has $3^n$ vertices, and is $U_5$-free (since $T_3$ is $U_5$-free and by induction, $G_{n-1}$ is $U_5$-free). We claim that $G_n$ has no transitive subtournament with more than $2^n$ vertices, which will prove the claim of tightness.

We proceed by induction on $n$. For $n=1$, $G_1$ is $T_3$, which has no transitive subtournament with more than 2 vertices. Suppose $n \ge 2$, and that $G_{n-1}$ has no transitive subtournament with more than $2^{n-1}$ vertices. Write $G_n = T_3(G_{n-1},G_{n-1},G_{n-1})$, and let $V(G_n) = V_1 \cup V_2 \cup V_3$ as in Definition~\ref{substitute}. Suppose $I$ is a transitive subset of $V(G_n)$. Then $I \cap V_i$ must be empty for some $i = 1$, 2, or 3, because otherwise $I$ would contain a cyclic triangle. Without loss of generality, assume $I \cap V_3 = \varnothing$. Now, by the inductive hypothesis, $| I \cap V_1 |, | I \cap V_2 | \le 2^{n-1}$. Thus, $| I | \le 2^{n-1} + 2^{n-1} = 2^n$, as desired.

We now prove the first part of the theorem. Let $G$ be a $U_5$-free tournament. We wish to show that $G$ has a transitive subtournament with at least $|V(G)|^\gamma$ vertices. We proceed by induction on $|V(G)|$.
The theorem clearly holds if $|V(G)| \le 1$. Now, let $|V(G)| = n \ge 2$, and assume the theorem holds for all tournaments $H$ with $|V(H)| < n$.

Since $|V(G)| > 1$, there exists a prime tournament $G^\prime$, $|V(G^\prime)| > 1$, along with an ordering $(v_1,\dotsc,v_m)$ of the vertices of $G^\prime$, such that $G = G^\prime(H_1,\dotsc,H_m)$ for some non-null $H_1$, \dots, $H_m$. Then $H_1$, \dots, $H_m$ are all $U_5$-free, and each has less than $n$ vertices (since $m > 1$). Thus, by the inductive hypothesis, each $H_i$ has a transitive subtournament with at least $|V(H_i)|^\gamma$ vertices. Hence, letting $V(G) = V_1 \cup \dotsb \cup V_m$ as in Definition~\ref{substitute}, we have that for each $i$, the set $V_i$ has a transitive subset with at least $|V_i|^\gamma$ vertices.

Now, since $G^\prime$ is prime and $U_5$-free, by Theorem~\ref{noU} either $G^\prime$ is $T_m$ for some odd $n$ or $V(G^\prime)$ can be partitioned into sets $X$, $Y$, and $Z$ such that $X \cup Y$, $Y \cup Z$, and $Z \cup X$ are transitive. We consider the two cases separately.

First, suppose that $V(G^\prime)$ can be partitioned into sets $X$, $Y$, and $Z$ such that $X \cup Y$, $Y \cup Z$, and $Z \cup X$ are transitive.
For each subset $S \subseteq V(G^\prime)$, define $S^G \subseteq V(G)$ by
\[
S^G = \bigcup_{v_i \in S} V_i.
\]
(Recall that $(v_1,\dotsc,v_m)$ was the ordering of $V(G^\prime)$ we used to define $G^\prime(H_1,\dotsc,H_m)$, and hence $v_i$ corresponds to $V_i$.) In particular, $X^G$, $Y^G$, and $Z^G$ partition $V(G)$. Without loss of generality, assume $|Z^G| = \min( |X^G|, |Y^G|, |Z^G|)$.

Now, since $X \cup Y$ is transitive, and each $V_i$ contains a transitive subset with at least $|V_i|^\gamma$ vertices, it follows that $G[(X \cup Y)^G]$ has a transitive subtournament with at least
\[
\sum_{v_i \in X \cup Y} |V_i|^\gamma
\] 
vertices. We have
\[
\sum_{v_i \in X \cup Y} |V_i|^\gamma = \sum_{v_i \in X} |V_i|^\gamma + \sum_{v_i \in Y} |V_i|^\gamma,
\]
and by Karamata's inequality applied to the concave function $f(x) = x^\gamma$ on $[0,\infty)$, we have
\begin{align*}
\sum_{v_i \in X} |V_i|^\gamma &\ge \left(\sum_{v_i \in X} |V_i|\right)^\gamma = |X^G|^\gamma, \\
\sum_{v_i \in Y} |V_i|^\gamma &\ge \left(\sum_{v_i \in Y} |V_i|\right)^\gamma = |Y^G|^\gamma.
\end{align*}
Hence, $G[(X \cup Y)^G]$ has a transitive subtournament with at least $|X^G|^\gamma + |Y^G|^\gamma$ vertices. Since $|Z^G| = \min( |X^G|, |Y^G|, |Z^G|)$ and $|X^G| + |Y^G| + |Z^G| = n$, it follows by Proposition~\ref{2outof3} that 
\[
|X^G|^\gamma + |Y^G|^\gamma \ge n^\gamma,
\]
and hence $G[(X \cup Y)^G]$ has a transitive subtournament with at least $n^\gamma$ vertices, which is as desired.

Now, suppose that $G^\prime$ is $T_m$ for some odd $m$. Since $|V(G^\prime)| > 1$, we have $m \ge 3$. Choose $(v_1,\dotsc,v_m)$ to be an ordering of $V(G^\prime)$ such that $v_i \ra v_j$ if $j-i \equiv 1$, 2, \dots, $(m-1)/2 \pmod{m}$. Without loss of generality, assume
\[
|V_1| = \max_i |V_i|.
\]
By taking the dual if necessary, we may also assume that
\[
|V_2| + \dotsb + |V_{(m+1)/2}| \ge |V_{(m+3)/2}| + \dotsb + |V_{m}|.
\]

Now, we consider two cases: $|V_1| \ge n/3$ and $|V_1| < n/3$.

\begin{case}
$|V_1| \ge n/3$
\end{case}
Since $\{v_1,\dots,v_{(m+1)/2}\}$ is a transitive set in $G^\prime$, and each $V_i$ has a transitive subset with at least $|V_i|^\gamma$ vertices, it follows that $G[V_1 \cup \dotsb \cup V_{(m+1)/2}]$ has a transitive subtournament with at least
\[
|V_1|^\gamma + |V_2|^\gamma + \dotsb + |V_{(m+1)/2}|^\gamma
\]
vertices. Applying Karamata's inequality, we have
\[
|V_2|^\gamma + \dotsb + |V_{(m+1)/2}|^\gamma \ge  \left( |V_2| + \dotsb + |V_{(m+1)/2}| \right)^\gamma
\]
and hence
\[
|V_1|^\gamma + |V_2|^\gamma + \dotsb + |V_{(m+1)/2}|^\gamma \ge |V_1|^\gamma + \left( |V_2| + \dotsb + |V_{(m+1)/2}| \right)^\gamma.
\]
We will show that the right side of this inequality is at least $n^\gamma$, which will complete the proof for this case. Since $|V_1| \ge n/3$, we have $|V_2| + \dotsb + |V_{m}| \le 2n/3$. Then since $|V_2| + \dotsb + |V_{(m+1)/2}| \ge |V_{(m+3)/2}| + \dotsb + |V_{m}|$, we have
\[
|V_{(m+3)/2}| + \dotsb + |V_{m}| \le n/3.
\]
Hence, we have both
\[
|V_1| \ge n/3 \ge |V_{(m+3)/2}| + \dotsb + |V_{m}| \quad \text{and} \quad |V_2| + \dotsb + |V_{(m+1)/2}| \ge |V_{(m+3)/2}| + \dotsb + |V_{m}|.
\]
Since 
\[
|V_1| + \left( |V_2| + \dotsb + |V_{(m+1)/2}| \right) + \left( |V_{(m+3)/2}| + \dotsb + |V_{m}| \right) = n,
\]
we thus have by Proposition~\ref{2outof3} that
\[
|V_1|^\gamma + \left( |V_2| + \dotsb + |V_{(m+1)/2}| \right)^\gamma \ge n^\gamma
\]
as desired.

\begin{case}
$|V_1| < n/3$
\end{case}
As in the previous case, $G[V_1 \cup \dotsb \cup V_{(m+1)/2}]$ has a transitive subtournament with at least
\begin{equation}\label{original}
|V_1|^\gamma + |V_2|^\gamma + \dotsb + |V_{(m+1)/2}|^\gamma
\end{equation}
vertices.
Let $|V_1| = a$. Let $|V_2| + \dotsb + |V_{(m+1)/2}| = qa+r$, where $q$, $r$ are nonnegative integers and $0 \le r < a$. Since $|V_1| = \max_i |V_i|$, we have $|V_i| \le a$ for all $i$. Hence, by Karamata's inequality, we have
\[
|V_2|^\gamma + \dotsb + |V_{(m+1)/2}|^\gamma \ge \underbrace{a^\gamma + \dotsb + a^\gamma}_{q \text{ times}} + r^\gamma.
\]
We claim that $r^\gamma \ge ra^{\gamma-1}$. Indeed, it holds for $r = 0$, and if $r > 0$, 
then since $\gamma - 1 < 0$ and $r < a$, we have $r^{\gamma-1} \ge a^{\gamma-1}$, and hence $r^\gamma \ge ra^{\gamma-1}$ as desired. Thus, we have
\begin{align*}
|V_2|^\gamma + \dotsb + |V_{(m+1)/2}|^\gamma &\ge \underbrace{a^\gamma + \dotsb + a^\gamma}_{q \text{ times}} + r^\gamma \\
&\ge qa^\gamma + ra^{\gamma-1} \\
&= a^{\gamma-1}(qa+r) \\
&= a^{\gamma-1}\left( |V_2| + \dotsb + |V_{(m+1)/2}| \right).
\end{align*}
Now, since $|V_2| + \dotsb + |V_{m}| = n-a$ and $|V_2| + \dotsb + |V_{(m+1)/2}| \ge |V_{(m+3)/2}| + \dotsb + |V_{m}|$, we have $|V_2| + \dotsb + |V_{(m+1)/2}| \ge (n-a)/2$. Thus,
\begin{align*}
|V_2|^\gamma + \dotsb + |V_{(m+1)/2}|^\gamma &\ge a^{\gamma-1}\left( |V_2| + \dotsb + |V_{(m+1)/2}| \right) \\
&\ge a^{\gamma-1}\left( \frac{n-a}{2} \right).
\end{align*}
Recalling expression \eqref{original}, it follows that $G[V_1 \cup \dotsb \cup V_{(m+1)/2}]$ has a transitive subtournament with at least
\begin{align*}
|V_1|^\gamma + |V_2|^\gamma + \dotsb + |V_{(m+1)/2}|^\gamma &\ge a^\gamma + a^{\gamma-1}\left( \frac{n-a}{2} \right) \\
&= \frac{1}{2} \left( a^\gamma + na^{\gamma-1} \right)
\end{align*}
vertices.

We claim that $(1/2)(a^\gamma + na^{\gamma-1}) \ge n^\gamma$, which will complete the proof. Let 
\[
g(x) = \frac{1}{2} \left( x^\gamma + nx^{\gamma-1} \right).
\]
Then
\begin{align*}
g^\prime(x) &= \frac{1}{2} \left( \gamma x^{\gamma-1} + n(\gamma-1)x^{\gamma-2} \right) \\
&= \frac{1}{2}x^{\gamma-2} \left( \gamma x + n(\gamma-1) \right).
\end{align*}
For $x < n(1-\gamma)/\gamma$, we have $\gamma x + n(\gamma-1) < 0$. Hence, $g^\prime(x) < 0$
for all $0 < x <  n(1-\gamma)/\gamma$, so $g$ is decreasing on $(0,n(1-\gamma)/\gamma)$. Now,
\[
\frac{n(1-\gamma)}{\gamma} = \frac{n(1-\log_3 2)}{\log_3 2} > \frac{n}{3}.
\]
Thus, $(0,n/3] \subseteq (0,n(1-\gamma)/\gamma)$, so the minimum value of $g$ on $(0,n/3]$ is $g(n/3) = n^\gamma$. Since $a \in  (0,n/3]$ by assumption, we have $(1/2)(a^\gamma + na^{\gamma-1}) \ge n^\gamma$, as desired.
\end{proof}

\begin{rem}
Using Theorem~\ref{noW} and arguments similar to the ones in this proof, one can show that every $W_5$-free tournament with $n$ vertices has a transitive subtournament with at least $n^{\log_7 3}$ vertices. This bound is tight, as can be seen by considering the tournaments $G_n$ defined by $G_1 = Q_7$ and $G_{n+1} = Q_7(G_n,G_n,\dotsc,G_n)$.
\end{rem}

\end{document}